\documentclass{amsart}

\usepackage{amsmath,amsthm,amssymb,bm}
\usepackage{hyperref}
\usepackage{a4wide}
\usepackage{cleveref}
\usepackage{graphicx,color}
\usepackage{tikz}
\numberwithin{equation}{section}
\usetikzlibrary{positioning}

\newtheorem{thm}{Theorem}[section]
\newtheorem{lem}[thm]{Lemma}
\newtheorem{prop}[thm]{Proposition}
\newtheorem{cor}[thm]{Corollary}
\theoremstyle{definition}
\newtheorem{exam}[thm]{Example}
\newtheorem{defn}[thm]{Definition}

\newtheorem{problem}[thm]{Problem}
\newtheorem{remark}[thm]{Remark}

\crefname{lem}{Lemma}{Lemmas}
\crefname{thm}{Theorem}{Theorems}
\crefname{prop}{Proposition}{Propositions}
\crefname{question}{Question}{Questions}
\crefname{defn}{Definition}{Definitions}
\crefname{conj}{Conjecture}{Conjectures}
\crefname{figure}{Figure}{Figures}
\crefname{cor}{Corollary}{Corollaries} 
\crefformat{equation}{(#2#1#3)}
\Crefformat{equation}{Equation #2(#1)#3}

\AddToHook{env/lem/begin}{\crefalias{thm}{lem}}
\AddToHook{env/prop/begin}{\crefalias{thm}{prop}}
\AddToHook{env/cor/begin}{\crefalias{thm}{cor}}
\AddToHook{env/exam/begin}{\crefalias{thm}{exam}}
\AddToHook{env/defn/begin}{\crefalias{thm}{defn}}
\AddToHook{env/conj/begin}{\crefalias{thm}{conj}}
\AddToHook{env/question/begin}{\crefalias{thm}{question}}
\AddToHook{env/problem/begin}{\crefalias{thm}{problem}}
\AddToHook{env/remark/begin}{\crefalias{thm}{remark}}

\newcommand\len{\operatorname{len}}
\newcommand\EE{\mathcal{E}}
\newcommand\JJ{\mathcal{J}}
\newcommand\Comp{\operatorname{Comp}}

\newcommand{\ZZ}{\mathbb{Z}}
\newcommand\wt{\operatorname{wt}}

\newcommand\LB[3]{\node[fill=white,draw=black,circle,inner sep=1pt] (L#2) at (0,4*#2/#3) {\( #1 \)};}
\newcommand\RB[3]{\node[fill=white,draw=black,circle,inner sep=1pt] (R#2) at (3,4*#2/#3) {\( #1 \)};}
\newcommand\drop[1]{\draw[ultra thick,in=0,out=90] (0) to (L#1);}
\newcommand\throw[1]{\draw[ultra thick,in=180,out=90] (0) to (R#1);}
\newcommand\edge[2]{\draw[ultra thick,in=180,out=0] (L#1) to (R#2);}
\newcommand\LW[3]{\node[left] at (-.3,4*#2/#3) {\( #1 \)};}
\newcommand\RW[3]{\node[right] at (3.3,4*#2/#3) {\( #1 \)};}

\title[Enumeration of multiplex juggling card sequences]{Enumeration of multiplex juggling card sequences using
  generalized \( q \)-derivatives}

\author{Yumin Cho}
\address{School of Biological Sciences, Seoul National University, Seoul, South Korea}
\email{ymcho07@snu.ac.kr}

\author{Jaehyun Kim}
\address{Department of Mechanical Engineering, Seoul National University, Seoul, South Korea}
\email{kjjhh2006@snu.ac.kr}

\author{Jang Soo Kim}
\address{Department of Mathematics,
Sungkyunkwan University (SKKU), Suwon, South Korea}
\email{jangsookim@skku.edu}

\author{Nakyung Lee}
\address{Department of Industrial Engineering, Seoul National University, Seoul, South Korea}
\email{nakyung06@snu.ac.kr}

\thanks{The authors are listed in alphabetical order, as is customary in mathematics, and contributed equally to this work.}


\begin{document}

\begin{abstract}
  In 2019, Butler, Choi, Kim, and Seo introduced a new type of
  juggling card that represents multiplex juggling patterns in a
  natural bijective way. They conjectured a formula for the generating
  function for the number of multiplex juggling cards with capacity
  \( 2 \). In this paper we prove their conjecture. More generally, we
  find an explicit formula for the generating function with any
  capacity. We also find an expression for the generating function for
  multiplex juggling card sequences by introducing a generalization of
  the \( q \)-derivative operator. As a consequence, we show that this
  generating function is a rational function.
\end{abstract}

\maketitle


\section{Introduction}

Juggling is an act of throwing and catching balls. Since the 1980s,
juggling has been studied mathematically by many researchers; for
example, see \cite{Ayyer2015, Banaian2016, Benedetti2020, Buhler1994,
  Butler2017, Butler2010, Cardinal2006, Chung2008, Ehrenborg1996,
  Engstrom2015, Stadler2002, Warrington2005} and references therein.
We refer the reader to \cite{Buhler1994,Polster2003} for the history
of mathematics of juggling.

Juggling can be divided into two categories: \emph{simple juggling}
and \emph{multiplex juggling}. In simple juggling, at most one ball is
caught and thrown at every beat. Multiplex juggling is a generalization
of simple juggling, where at most \( k \) balls are caught and thrown
at every beat. The number \( k \) is called the \emph{(hand)
  capacity}.

A \emph{simple juggling pattern} is a bijection \( f:\ZZ\to \ZZ \)
such that \( f(i)\ge i \) for all \( i\in \ZZ \). This can be
understood as the situation that a juggler catches a ball at beat
\( i \) and throws it immediately so that it lands at beat \( f(i) \).
(If \( f(i)=i \), then the juggler does not catch or throw any ball at
beat \( i \).) A simple juggling pattern is \emph{periodic} if the
function \( h:\ZZ\to \ZZ \) defined by \( h(i)=f(i)-i \) is periodic.

Buhler, Eisenbud, Graham, and Wright \cite{Buhler1994} showed that the
number of simple juggling patterns with \( b \) balls and period
\( p \) is equal to \( (b+1)^p - b^p \). Ehrenborg and Readdy
\cite{Ehrenborg1996} found a simple proof of this result by
introducing juggling cards.

There are some models of multiplex juggling cards introduced by
Ehrenborg--Readdy \cite{Ehrenborg1996} and
Butler--Chung--Cummings--Graham \cite{Butler2017}. However, they do
not represent multiplex juggling patterns in a bijective way.
In 2019, Butler, Choi, Kim, and Seo \cite{Butler2019} introduced a
new type of card that represents multiplex juggling patterns in a natural
bijective way. We refer the reader to \cite{Butler2019} for more
details.

In this paper we study enumerative properties of the number of
multiplex juggling card sequences. We consider three natural
parameters of multiplex juggling card sequences: the number \( b \) of
balls, the capacity \( k \), and the length \( \ell \) of a sequence
of cards. Let \( J(b,k,\ell) \) denote the number of multiplex
juggling card sequences with given parameters \( k,b \), and
\( \ell \). See \Cref{sec:preliminaries} for the precise definition.

Now we review some known results on the number \( J(b,k,1) \) of
multiplex juggling cards with \( b \) balls and capacity \( k \). If
\( k=1 \), it is immediate from the definition that
\begin{equation}\label{eq:6}
   \sum_{b\ge 0}  J(b,1,1) x^b
   = \sum_{b\ge0} (b+1) x^b = \frac{1}{(1-x)^2}.
\end{equation}
For the other extreme case \( k=\infty \), Butler et
al.~\cite[Theorem~4]{Butler2019} showed that the sequence
\( \{ J(b,\infty,1) \}_{b\ge0} \) satisfies the following simple
linear recurrence: \( J(0,\infty,1) = 1 \), \( J(1,\infty,1) = 2 \),
\( J(2,\infty,1) = 7 \), and for \( b\ge3 \),
\begin{equation}\label{eq:bulter}
  J(b,\infty,1) = 4J(b-1,\infty,1) - 2J(b-2,\infty,1). 
\end{equation}
By the standard method for linear recurrences, see
\cite[Theorem~4.1.1]{EC1}, one can easily check that the recursion~\eqref{eq:bulter}
is equivalent to
\begin{equation}\label{eq:2}
  \sum_{b\ge0} J(b,\infty,1) x^b = \frac{1-2x+x^2}{1-4x+2x^2}.
\end{equation}
For the case \( k=2 \), Butler et al. \cite[Conjecture~13]{Butler2019}
conjectured the following generating function formula:
 \begin{equation}\label{eq:conj}
   \sum_{b\ge 0}  J(b,2,1) x^b
   = \frac{1-x+x^2+x^3}{(1-x-x^2)^3}.
 \end{equation} 

 In this paper we prove this conjecture; see \Cref{exa:3}. More
 generally, we find an explicit formula for the generating function
 \( \sum_{b\ge 0} J(b,k,1) x^b \) for any capacity \( k \); see
 \Cref{cor:l=1}. We also find an expression for
 \begin{equation}\label{eq:11}
   \sum_{b\ge 0} J(b,k,\ell) x^b
 \end{equation}
 for any capacity \( k \) and length \( \ell \) by introducing a
 generalization of the \( q \)-derivative operator; see \Cref{thm:3}.
 As a consequence, we show that the generating function \eqref{eq:11} is a rational function
 in \( x \).

 The rest of this paper is organized as follows. In
 \Cref{sec:preliminaries} we provide the necessary definitions. In
 \Cref{sec:gener-funct-numb} we study the generating function
 \eqref{eq:11} for the case \( \ell=1 \). In
 \Cref{sec:generating-function-kl} we study the generating function
 \eqref{eq:11} for a general \( \ell \). In \Cref{sec:rationality} we
 show that \eqref{eq:11} is a rational function in \( x \). In
 \Cref{sec:conclusion} we summarize our results and propose some open
 problems.

\section{Preliminaries}
\label{sec:preliminaries}

In this section we give necessary definitions. Throughout this paper
we will use the notation \( [n] = \{ 1,\dots,n\} \) for a positive
integer \( n \).

\begin{defn}
  A \emph{composition} is a sequence
  \( \alpha=(\alpha_1,\dots,\alpha_r) \) of positive integers. Each
  \( \alpha_i \) is called a \emph{part} of~\( \alpha \). The
  \emph{size} \( |\alpha| \) of \( \alpha \) is defined by
  \( |\alpha| = \alpha_1 + \cdots + \alpha_r \). If \( |\alpha|=n \),
  we say that \( \alpha \) is a composition of \( n \). The
  \emph{length} \( \ell(\alpha) \) of \( \alpha \) is defined to be
  the number of parts in \( \alpha \). We denote by \( \Comp(n) \) the
  set of compositions of~\( n \). We also denote by \( \Comp(n,k) \)
  the set of compositions of~\( n \) with \( k \) parts.
\end{defn}

\begin{defn}\label{def:2}
  A \emph{(multiplex juggling) card} is a triple
  \( (\alpha,\beta,f) \) such that
  \( \alpha=(\alpha_1,\dots,\alpha_r) \) and
  \( \beta=(\beta_1,\dots,\beta_s) \) are compositions with
  \( |\alpha|=|\beta| \), and \( f:[r]\to \{0\}\cup[s] \) is a
  strictly increasing function satisfying
  \( \alpha_i \le \beta_{f(i)} \) for all \( i\in [r] \) with
  \( f(i)\ne 0 \). We call \( \alpha \) and \( \beta \) the
  \emph{arrival composition} and the \emph{departure composition} of
  the card, respectively. If every part of \( \alpha \) and
  \( \beta \) is at most \( k \), then we say that the card has
  \emph{capacity}~\( k \). We also say that the card has
  \( |\alpha| \) \emph{balls}.
\end{defn}

We can visualize a card \( (\alpha,\beta,f) \) as follows. Suppose
\( \alpha=(\alpha_1,\dots,\alpha_r) \) and
\( \beta=(\beta_1,\dots,\beta_s) \). Consider a rectangle with \( r \)
vertices on the left side labeled \( \alpha_1,\dots,\alpha_r \) from
bottom to top, \( s \) vertices on the right side labeled
\( \beta_1,\dots,\beta_s \) from bottom to top, and a vertex on the
bottom side called the \emph{ground vertex}. If \( f(1)\ne 0 \), then
it follows from \Cref{def:2} that we must have \( r=s \) and
\( f(i)=i \) for all \( i\in [r] \). In this case, draw a curve from
vertex \( \alpha_i \) to vertex~\( \beta_i \) for all
\( i\in [r] \). If \( f(1) = 0 \), then draw a curve from vertex
\( \alpha_1 \) to the ground vertex, a curve from vertex
\( \alpha_i \) to vertex \( \beta_{f(i)} \) for each
\( i\in \{2,3,\dots,r\} \), and a curve from the ground vertex to each
vertex \( \beta_j \) such that either \( j = f(i) \) for some
\( i\in [r] \) with \( \alpha_i < \beta_j \) or \( j \) is not in the
image of \( f \).

Such a visualization can be understood as \( \alpha_i \) balls
entering the card on level \( i \) and \( \beta_j \) balls leaving the
card on level \( j \). Note that if \( f(1)\ne 0 \) then all balls
stay in the air, and if \( f(1)= 0 \) then \( \alpha_1 \) balls are
caught and thrown again, so that these balls are redistributed.

\begin{exam}\label{exa:1}
  Let \( (\alpha,\beta,f) \) be the card such that
  \( \alpha = (4,2,3) \), \( \beta = (4,2,3) \), and
  \( f:\{1,2,3\}\to\{0,1,2,3\} \) is the function given by
  \( f(1)=1 \), \( f(2)=2 \), and \( f(3)=3 \). Then the card can be
  visualized as the left diagram in \Cref{fig:visualization}.
\end{exam}

\begin{exam}\label{exa:2}
  Let \( (\alpha,\beta,f) \) be the triple such that
  \( \alpha = (6,1,2,2) \), \( \beta = (3,1,2,3,2) \), and
  \( f:\{1,2,3,4\}\to\{0,1,2,3,4,5\} \) is the function given by
  \( f(1)=0 \), \( f(2)=1 \), \( f(3)=3 \), and \( f(4)=4 \). Then the
  card can be visualized as the right diagram in
  \Cref{fig:visualization}.
\end{exam}

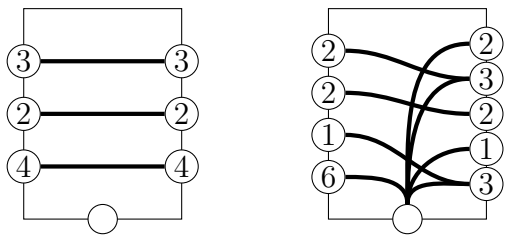
\begin{figure}
     \centering
\begin{tikzpicture}
\begin{scope}[scale=.7]
  \draw (0,0) rectangle (3,4);
  \node[fill=white,draw=black,circle] (0) at (1.5,0) {};
  \LB414 \LB224 \LB334
  \RB414 \RB224 \RB334
  \edge11 \edge22 \edge33
  \LW{}14 \LW{}24 \LW{}34
  \RW{}14 \RW{}24 \RW{}34
\end{scope}
\end{tikzpicture} \qquad 
\begin{tikzpicture}
\begin{scope}[scale=.7]
  \draw (0,0) rectangle (3,4);
  \node[fill=white,draw=black,circle] (0) at (1.5,0) {};
  \LB615 \LB125 \LB235 \LB245 \RB316 \RB126 \RB236 \RB346 \RB256
  \edge21 \edge33 \edge44 \drop1 \throw1 \throw2 \throw4 \throw5
  \LW{}15 \LW{}25 \LW{}35 \LW{}45
  \RW{}16 \RW{}26 \RW{}36 \RW{}46 \RW{}56
\end{scope}
\end{tikzpicture}
\caption{The left diagram is a visualization of the card in
  \Cref{exa:1}. The right diagram is a visualization of the card in
  \Cref{exa:2}.}
     \label{fig:visualization}
 \end{figure}

 One multiplex juggling card represents a situation of multiplex
 juggling at a given moment, say at beat \( i \). We say that two
 juggling cards \((\alpha, \beta, f)\) and
 \(\left(\alpha^{\prime}, \beta^{\prime}, f^{\prime}\right)\) are
 \emph{compatible} if \(\beta=\alpha^{\prime}\). By listing \( \ell \)
 compatible cards, we can represent a situation of multiplex juggling
 from beats \( 1 \) to \( \ell \). To be more precise, we introduce
 the following definition.

\begin{defn}\label{def:l-cs}
  An \emph{\( \ell\)-card sequence} with \( b \) \emph{balls} and
  \emph{capacity} \( k \) is a sequence
  \( (C_1,\dots, C_\ell) \) of \( \ell \) cards with
  \( b \)~balls and capacity \( k \) such that the departure
  composition of \( C_i \) is equal to the arrival
  composition of \( C_{i+1} \) for all \( i\in [\ell-1] \).
  We denote by \( \JJ(b,k,\ell) \) the set of \( \ell \)-card
  sequences with \( b \) balls and capacity \( k \). We also define
  \( J(b,k,\ell) = |\JJ(b,k,\ell)| \).
\end{defn}

For example, see \Cref{fig:card-seq}.

\begin{figure}
    \centering
\begin{tikzpicture}
\begin{scope}[scale=.7]
  \draw (0,0) rectangle (3,4);
  \node[fill=white,draw=black,circle] (0) at (1.5,0) {};
  \LB414 \LB224 \LB334 \RB414 \RB324 \RB234
  \edge21 \edge32 \drop1 \throw1 \throw3
\end{scope}
\end{tikzpicture}\qquad 
\begin{tikzpicture}
\begin{scope}[scale=.7]
  \draw (0,0) rectangle (3,4);
  \node[fill=white,draw=black,circle] (0) at (1.5,0) {};
  \LB414 \LB324 \LB234 \RB215 \RB325 \RB335 \RB145
  \edge22 \edge33 \drop1 \throw1 \throw3 \throw4
\end{scope}
\end{tikzpicture}\qquad 
\begin{tikzpicture}
\begin{scope}[scale=.7]
  \draw (0,0) rectangle (3,4);
  \node[fill=white,draw=black,circle] (0) at (1.5,0) {};
  \LB215 \LB325 \LB335 \LB145 \RB215 \RB325 \RB335 \RB145
  \edge11 \edge22 \edge33 \edge44
\end{scope}
\end{tikzpicture}\qquad 
\begin{tikzpicture}
\begin{scope}[scale=.7]
  \draw (0,0) rectangle (3,4);
  \node[fill=white,draw=black,circle] (0) at (1.5,0) {};
  \LB215 \LB325 \LB335 \LB145 \RB415 \RB325 \RB135 \RB145
  \edge21 \edge32 \edge43 \drop1 \throw1 \throw4
\end{scope}
\end{tikzpicture}
\caption{An example of a 4-card sequence.}
     \label{fig:card-seq}
\end{figure}
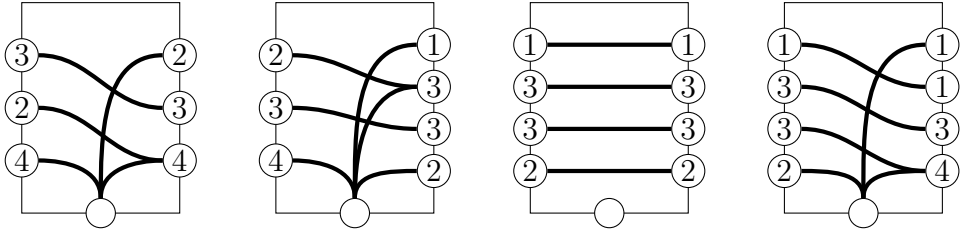

\section{Juggling cards with fixed capacity}
\label{sec:gener-funct-numb}

In this section we provide three expressions for the generating
function
\begin{equation}\label{eq:7}
  \sum_{b\ge0} J(b,k,1)x^b 
\end{equation}
with fixed capacity \( k \). To this end we give another description
of a card using embeddings introduced in \cite{Butler2019}.

Recall that \( \JJ(b,k,1) \) is the set of all cards with \( b \)
balls and capacity \( k \). We use the notation \( r^s\) to denote
the word consisting of \( s \) \( r \)'s.

\begin{defn}
  A \emph{\( (b,k) \)-embedding} is a sequence
  \( \gamma = (\gamma_1,\dots,\gamma_s) \) satisfying the following
  conditions:
  \begin{itemize}
  \item Each \( \gamma_i \) is a word of the form
    \( \gamma_i= 0^u 1^v \) for some integers \( u,v\ge0 \) with \( 1\le u+v\le k \).
  \item The total number of \( 1 \)'s in all of \( \gamma_1,\dots,\gamma_s \)
    is at most \( k \).
  \item The sum of the lengths of \( \gamma_i \) for all \( i\in [s] \) is equal to \( b \).
  \end{itemize}
  Let \( \EE(b,k) \) be the set of \( (b,k) \)-embeddings.
\end{defn}

Let \( (\alpha,\beta,f)\in \JJ(b,k,1) \), where
\( \alpha=(\alpha_1,\dots,\alpha_r) \) and
\( \beta= (\beta_1,\dots,\beta_s) \). Let
\( \gamma = (\gamma_1,\dots,\gamma_s) \) be the \( (b,k) \)-embedding
defined by
\[
  \gamma_j =
  \begin{cases}
   1^{\beta_j} & \mbox{if \( j \) is not in the image of \( f \)},\\
   0^{\alpha_i}1^{\beta_j-\alpha_i} & \mbox{if \( j = f(i) \)}.
  \end{cases}
\]
This can also be understood using the visualization of the card
\( (\alpha,\beta,f) \) as follows. For each vertex \( \beta_j \), if
it is connected to vertex \( \alpha_i \) for some \( i\in [r] \),
then \( \gamma_j = 0^{\alpha_i}1^{\beta_j-\alpha_i} \) and otherwise
\( \gamma_j = 1^{\beta_j} \). The \( 0 \)'s encode the balls passing
over the ground vertex, whereas the \( 1 \)'s encode the balls that were
thrown up from this vertex.

\begin{exam}
  If \( (\alpha,\beta,f) \) is the card in \Cref{exa:1}, then
  \( \gamma=(0000,00,000) \). If \( (\alpha,\beta,f) \) is the card in
  \Cref{exa:2}, then \( \gamma=(011,1,00,001,11) \). See
  \Cref{fig:def_alpha_card}.
\end{exam}

\begin{figure}
     \centering
\begin{tikzpicture}
\begin{scope}[scale=.7]
  \draw (0,0) rectangle (3,4);
  \node[fill=white,draw=black,circle] (0) at (1.5,0) {};
  \LB414 \LB224 \LB334
  \RB414 \RB224 \RB334
  \edge11 \edge22 \edge33
  \LW{}14 \LW{}24 \LW{}34
  \RW{0000}14 \RW{00}24 \RW{000}34
\end{scope}
\end{tikzpicture} \qquad  \qquad 
\begin{tikzpicture}
\begin{scope}[scale=.7]
  \draw (0,0) rectangle (3,4);
  \node[fill=white,draw=black,circle] (0) at (1.5,0) {};
  \LB615 \LB125 \LB235 \LB245 \RB316 \RB126 \RB236 \RB346 \RB256
  \edge21 \edge33 \edge44 \drop1 \throw1 \throw2 \throw4 \throw5
  \LW{}15 \LW{}25 \LW{}35 \LW{}45
  \RW{011}16 \RW{1}26 \RW{00}36 \RW{001}46 \RW{11}56
\end{scope}
\end{tikzpicture}
\caption{The left diagram shows the card in \Cref{exa:1} and its
  corresponding embedding \( (0000,00,000) \). The right diagram shows
  the card in \Cref{exa:2} and its corresponding embedding
  \( (011,1,00,001,11) \).}
     \label{fig:def_alpha_card}
 \end{figure}
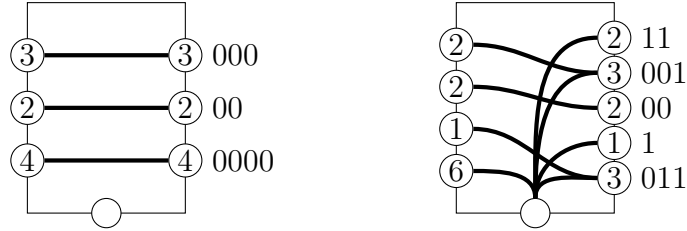

It is easy to see that the map \( (\alpha,\beta,f)\mapsto \gamma \) is
a bijection from \( \JJ(b,k,1) \) to \( \EE(b,k) \). Therefore, we can
identify a card with \( b \) balls and capacity \( k \) with a
\( (b,k) \)-embedding.

Now we are ready to find an expression for the generating function in
\eqref{eq:7}. For a formal power series \( F(z) \) in \( z \), the
notation \( [z^k]F(z) \) denotes the coefficient of \( z^k \) in
\( F(z) \).

\begin{prop}\label{prop:1}
  For a positive integer \( k \), we have
  \[
    \sum_{b\ge0} J(b,k,1)x^b 
    = [z^k] \left( \frac{1}{1-z} \cdot \frac{1}{1-\sum_{i=1}^k x^{i} \sum_{j=0}^{i} z^j } \right).
\]
\end{prop}

\begin{proof}
  Since \( J(0,k,1)=1 \) and
  \( J(b,k,1) = |\JJ(b,k,1)| = |\EE(b,k)| \) for \( b\ge1 \), we can
  instead consider the \( (b,k) \)-embeddings for all \( b\ge1 \).
  Recall that if \( \gamma = (\gamma_1,\dots,\gamma_s) \) is a
  \( (b,k) \)-embedding, then \( \gamma_i= 0^{u_i} 1^{v_i} \) for some
  integers \( u_i,v_i\ge0 \) with \( 1\le u_i+v_i\le k \) such that
  \( 0\le v_1 + \cdots + v_s \le k \) and
  \( u_1 + \cdots + u_s + v_1 + \cdots + v_s = b \).

  Let \( W \) be the set of words \( 0^u1^v \) such that \( u,v\ge0 \)
  and \( 1\le u+v\le k \). For a word \( w=0^u1^v\in W \), let
  \( \ell(w)=u+v \) and \( \ell_1(w) = v \). By definition, we have
  \[
    \sum_{w\in W} x^{\ell(w)} z^{\ell_1(w)} = \sum_{i=1}^k x^{i} \sum_{j=0}^{i} z^j.
  \]
  Thus we obtain
  \begin{equation}\label{eq:3}
    \frac{1}{1-\sum_{i=1}^k x^{i} \sum_{j=0}^{i} z^j } = 
    1+ \sum_{t\ge1} \sum_{(w_1,\dots,w_t)\in W^t} x^{\ell(w_1) + \cdots + \ell(w_t)}
    z^{\ell_1(w_1) + \cdots + \ell_1(w_t)}.
  \end{equation}

  On the other hand, for \( b\ge1 \), \( \EE(b,k) \) is the set of
  \( \gamma = (\gamma_1,\dots,\gamma_t)\in W^t \), \( t\ge1 \), such
  that \( \ell(\gamma_1) + \cdots + \ell(\gamma_t) = b \) and
  \( 0\le \ell_1(\gamma_1) + \cdots + \ell_1(\gamma_t) \le k \). Thus
  \begin{align*}
    \sum_{b\ge0} J(b,k,1)x^b 
    &= 1+ \sum_{b\ge1} |\EE(b,k)|x^b \\
    &= 1+ \sum_{t\ge1} \sum_{\substack{(w_1,\dots,w_t)\in W^t \\ 0\le \ell_1(\gamma_1) + \cdots + \ell_1(\gamma_t) \le k}}
    x^{\ell(w_1) + \cdots + \ell(w_t)}\\
    &= \sum_{p=0}^{k} [z^p]\left( 1+ \sum_{t\ge1} \sum_{(w_1,\dots,w_t)\in W^t}
    x^{\ell(w_1) + \cdots + \ell(w_t)}
    z^{\ell_1(w_1) + \cdots + \ell_1(w_t)} \right).
  \end{align*}
  By equation \eqref{eq:3}, this expression is equal to
  \begin{equation}\label{eq:15}
 \sum_{p=0}^{k} [z^p]
      \left( \frac{1}{1-\sum_{i=1}^k x^{i} \sum_{j=0}^{i} z^j } \right)
    =[z^k] \left( \frac{1}{1-z} \cdot \frac{1}{1-\sum_{i=1}^k x^{i} \sum_{j=0}^{i} z^j } \right),
  \end{equation}
 as desired. 
\end{proof}

\begin{remark}
  \Cref{prop:1} can be used to prove equation \eqref{eq:2}, which is equivalent
  to the result of Butler et al. \cite[Theorem~4]{Butler2019}. To see
  this, note that, by equation \eqref{eq:15}, \Cref{prop:1} can be rewritten as
  \[
    \sum_{b\ge0} J(b,k,1)x^b 
    = \sum_{p=0}^k [z^p] \left(  \frac{1}{1-\sum_{i=1}^k x^{i} \sum_{j=0}^{i} z^j } \right).
\]
Taking the limit as \( k \) tends to infinity, we have
\[
  \sum_{b\ge0} J(b,\infty,1)x^b = \sum_{p=0}^\infty [z^p] \left(
    \frac{1}{1-\sum_{i=1}^\infty x^{i} \sum_{j=0}^{i} z^j } \right).
\]
Adding the coefficient of \( z^p \) for all nonnegative \( p \)
is equivalent to substituting \( z=1 \). Hence
we obtain
\[
  \sum_{b\ge0} J(b,\infty,1)x^b = \frac{1}{1-\sum_{i=1}^\infty x^{i}
    (i+1) }
  = \frac{1}{1- \frac{d}{dx}\left( \frac{1}{1-x}-1-x \right)}
  = \frac{1-2x+x^2}{1-4x+2x^2}.
\]
The sequence \( \{J(b,\infty,1)\}_{b\ge0} \) is
\href{https://oeis.org/A003480}{A003480} in the On-Line Encyclopedia
of Integer Sequences (OEIS) \cite{OEIS}:
 \[
   1, 2, 7, 24, 82, 280, 956, 3264, 11144, 38048, 129904, 443520, 1514272, \dots.
 \]
\end{remark}

Now we give another expression for the generating function in
\eqref{eq:7}. For a composition \( \alpha\in \Comp(n) \), we define
\( \ell_2(\alpha) \) to be the number of parts of \( \alpha \) at
least \( 2 \).

\begin{thm}\label{thm:l=1}
  For a positive integer \( k \), we have
\[
    \sum_{b\ge0} J(b,k,1)x^b 
    = \sum_{\alpha\in \Comp(k)} 
    \frac{(-1)^{\ell_2(\alpha)} x^{k-\ell(\alpha)}}{(1-x - \cdots - x^{k})^{1+\ell(\alpha)}}.
\]
\end{thm}

\begin{proof}
We will modify the right-hand side of \Cref{prop:1}.
We have
\begin{align*}
     \frac{1}{1-z} \cdot \frac{1}{1-\sum_{i=1}^k x^{i} \sum_{j=0}^{i} z^j } 
    &=   \frac{1}{1-z-
      \sum_{i=1}^k x^{i} \sum_{j=0}^{i} z^j + \sum_{i=1}^k x^{i} \sum_{j=0}^{i} z^{j+1} } \\
    &=   \frac{1}{1-z+
      \sum_{i=1}^k x^{i} (z^{i+1}-1)  } \\
    &=   \frac{1}{1-\sum_{i=1}^k x^{i}-z\left(1-\sum_{i=1}^k x^{i} z^{i}\right)} \\
    &= \frac{1}{1-\sum_{i=1}^k x^{i}}\cdot
      \frac{1}{1-z\left(1-\sum_{i=1}^k x^{i} z^{i}\right)
      \left(1-\sum_{i=1}^k x^{i}\right)^{-1}} \\
    &= \sum_{r\ge0} 
      \frac{z^r\left(1-\sum_{i=1}^k x^{i} z^{i}\right)^r}
      {\left(1-\sum_{i=1}^k x^{i}\right)^{1+r}}.
\end{align*}
Therefore by \Cref{prop:1}, we obtain
\begin{equation}\label{eq:1}
      \sum_{b\ge0} J(b,k,1)x^b 
    = [z^k] \left(
\sum_{r\ge0} \frac{(z-xz^2 - \cdots - x^kz^{k+1})^r}
      {(1-x - \cdots - x^{k})^{1+r}} \right).
\end{equation}

Observe that
\[
  (z-xz^2 - \cdots - x^kz^{k+1})^r 
  = \sum_{\alpha_1,\dots,\alpha_r\in [k+1]} z^{\alpha_1 + \cdots + \alpha_r}
  (-1)^{\ell_2(\alpha_1,\dots,\alpha_r)} x^{\alpha_1 + \cdots + \alpha_r - r}.
\]
Hence we can rewrite \eqref{eq:1} to obtain the result of the theorem.
\end{proof}

For small values of \( k \), one can easily find explicit formulas
for the generating function in \eqref{eq:7} using \Cref{thm:l=1} as
follows.

\begin{exam}
    If \( k=1 \) then \( \Comp(k) = \{(1)\} \). Thus we have
  \[
    \sum_{b\ge0} J(b,1,1)x^b = \frac{1}{(1-x)^2}.
  \]
\end{exam}

\begin{exam}\label{exa:3}
    If \( k=2 \) then \( \Comp(k) = \{(2),(1,1)\} \). Thus we have
  \[
   \sum_{b\ge 0}  J(b,2,1) x^b
   = \frac{-x}{(1-x-x^2)^2} + \frac{1}{(1-x-x^2)^3}
   = \frac{1-x+x^2+x^3}{(1-x-x^2)^3}.
 \]
 This proves the identity \eqref{eq:conj} conjectured by Butler et
 al.~\cite[Conjecture~13]{Butler2019}. The sequence
 \( \{J(b,2,1)\}_{b\ge0} \) is \href{https://oeis.org/A370304}{A370304}
in OEIS \cite{OEIS}:
 \[
   1, 2, 7, 17, 41, 91, 195, 403, 812, 1601, 3102, 5922, 11165, 20824, 38477,\ldots.
 \]
\end{exam}

\begin{exam}
    If \( k=3 \) then \( \Comp(k) = \{(3),(2,1),(1,2), (1,1,1)\} \). Thus we have
  \begin{align*}
    \sum_{b\ge 0}  J(b,3,1) x^b
    &= \frac{-x^2}{(1-x-x^2-x^3)^2} + 
      \frac{-2x}{(1-x-x^2-x^3)^3} +
      \frac{1}{(1-x-x^2-x^3)^4}\\
    &= \frac{1-2x+x^2+4x^3+3x^4-3x^6-2x^7-x^8}{(1-x-x^2-x^3)^4}.
 \end{align*}
The sequence \( \{J(b,3,1)\}_{b\ge0} \) is
\href{https://oeis.org/A370306}{A370306} in OEIS \cite{OEIS}:
 \[
1, 2, 7, 24, 70, 198, 532, 1370, 3418, 8296, 19677, 45770, 104687, 235972,\ldots.
 \]
\end{exam}

In fact, using \Cref{thm:l=1}, we can find an explicit formula for the
generating function in \eqref{eq:7} for a general capacity \( k \).

\begin{cor}\label{cor:l=1}
  For a positive integer \( k \), we have
  \[
    \sum_{b\ge0} J(b,k,1)x^b 
    = \sum_{r=1}^{k} \sum_{s=0}^{r}
    \frac{(-1)^{r-s} \binom{r}{s} \binom{k-r-1}{r-s-1} x^{k-r}}{(1-x - \cdots - x^{k})^{1+r}} ,
\]
where we extend the binomial coefficient by \( \binom{-1}{n} = \binom{n}{-1} =1\)
if \( n=-1 \) and \( \binom{-1}{n} = \binom{n}{-1} =0\) otherwise.
\end{cor}

\begin{proof}
We can rewrite \Cref{thm:l=1} as
\begin{equation}\label{eq:4}
      \sum_{b\ge0} J(b,k,1)x^b 
      =   \sum_{r=1}^{k} \frac{ x^{k-r}}{(1-x - \cdots - x^{k})^{1+r}}
      \sum_{\alpha\in \Comp(k,r)} (-1)^{\ell_2(\alpha)} .
\end{equation}
It is easy to see that the number of \( \alpha\in \Comp(k,r) \) with
exactly \( s \) parts equal to \( 1 \) is
\( \binom{r}{s} \binom{k-r-1}{r-s-1} \).
Hence, 
\begin{equation}\label{eq:5}
  \sum_{\alpha\in \Comp(k,r)} (-1)^{\ell_2(\alpha)}= \sum_{s=0}^{r}
  (-1)^{r-s} \binom{r}{s} \binom{k-r-1}{r-s-1}.
\end{equation}
By combining equations \eqref{eq:4} and \eqref{eq:5} we obtain the desired formula.
\end{proof}

\begin{remark}
Note that the right-hand side of \eqref{eq:5}, say
\[
  C(k,r) = \sum_{s=0}^{r} (-1)^{r-s} \binom{r}{s} \binom{k-r-1}{r-s-1}
\]
can be defined for any nonnegative integers \( k \) and \( r \).
Since \( \binom{-n}{k} = (-1)^{k} \binom{n+k-1}{k} \), we have, for
\( k<r \),
\[
  C(k,r) = - \sum_{s=0}^{r} \binom{r}{s} \binom{2r-k-s-1}{r-s-1}
 = - \sum_{s=1}^{r} \binom{r}{s} \binom{r-k-1}{s-1}.
\]
This shows that, for \( 1\le k<r \), we have
\( C(k,r) = -A(r,r-k+1) \), where \( A(n,k) \) is the sequence
\href{https://oeis.org/A050143}{A050143} in OEIS \cite{OEIS}. We note
that the sequence \( A(n,k) \) has an interpretation using certain
lattice paths; see \cite{OEIS}.
\end{remark}

\section{Juggling card sequences with fixed capacity}
\label{sec:generating-function-kl}

In this section we find an expression for the generating function for
\( J(b,k,\ell) \) with fixed capacity \( k \) and length~\( \ell \).
To this end we introduce \( (b,k,\ell) \)-embeddings, which generalize
the notion of \( (b,k) \)-embeddings of cards to \( \ell \)-card
sequences. In order to analyze \( (b,k,\ell) \)-embeddings, we then
introduce an operator which generalizes the \( q \)-derivative
operator.

Observe that the \( (b,k) \)-embedding of a card keeps track of when
balls are thrown. In what follows we extend this notion to
\( \ell \)-card sequences.

Let \( (C_1,\dots,C_\ell)\in J(b,k,\ell) \). For
\( i = 0,1,\dots,\ell \), we will define a sequence \( \alpha^{(i)} \) of
words of the form \( 0^{c_1} 1^{c_1} \cdots i^{c_i} \) as follows.

\begin{itemize}
\item First, we define \( \alpha^{(0)}=(0^{a_1},\dots,0^{a_s}) \), where
  \( (a_1,\dots,a_s) \) is the departure composition of \( C_1 \).
\item For \( i\in [\ell] \), suppose \( (\beta_1,\dots,\beta_t) \) is
  the \( (b,k) \)-embedding of \( C_i \). Then each \( \beta_j \) is a
  word of the form \( 0^u1^v \).
\begin{description}
\item[Case 1] There are no \( 1 \)'s in
  \( \beta_1,\dots,\beta_t \). In this case we define
  \( \alpha^{(i)} = \alpha^{(i-1)} \).
\item[Case 2] There is at least one \( 1 \) in
\( \beta_1,\dots,\beta_t \).
 Then, first, we replace each \( 1 \) by
\( i \) in \( \beta_1,\dots,\beta_t \). Let
\( \beta_{d_1},\dots,\beta_{d_m} \) be the words among
\( \beta_1,\dots,\beta_t \) containing at least one \( 0 \), where
\( d_1<\dots<d_m \). Then \( \alpha^{(i-1)} \) must have \( m+1 \)
words. Let
\( \alpha^{(i-1)} = (\alpha^{(i-1)}_1,\dots,\alpha^{(i-1)}_{m+1}) \).
For each \( j\in [m] \), we replace the subword of \( \beta_{d_j} \)
consisting of zeros by \( \alpha^{(i-1)}_{j+1} \).
\end{description}
\end{itemize}

\begin{exam}
  Let \( (C_1,\dots,C_4) \) be the \( 4 \)-card sequence in
  \Cref{fig:card-seq}. Then we have \( \alpha^{(0)}=(0000,00,000) \),
  \( \alpha^{(1)}=(0011,000,11) \), \( \alpha^{(2)}=(22,000,112,2) \),
  \( \alpha^{(3)}=(22,000,112,2) \), and
  \( \alpha^{(4)}=(0004,112,2,4) \). See \Cref{fig:embedding2}.
\end{exam}

Note that a 0 in \(\alpha^{(i)}\) means that the associated ball came
from the original set of balls whereas a \(j\) in \(\alpha^{(i)}\) means
that the associated ball was thrown up by card \(C_j\).

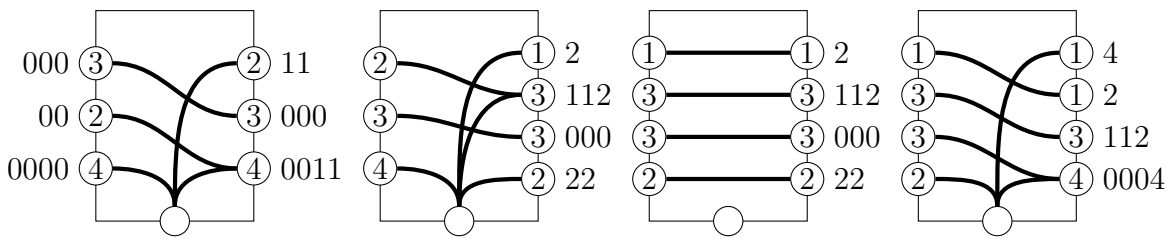
\begin{figure}
    \centering
\begin{tikzpicture}
\begin{scope}[scale=.7]
  \draw (0,0) rectangle (3,4);
  \node[fill=white,draw=black,circle] (0) at (1.5,0) {};
  \LB414 \LB224 \LB334 \RB414 \RB324 \RB234
  \edge21 \edge32 \drop1 \throw1 \throw3
  \LW{0000}14 \LW{00}24 \LW{000}34
  \RW{0011}14 \RW{000}24 \RW{11}34
\end{scope}
\end{tikzpicture}
\begin{tikzpicture}
\begin{scope}[scale=.7]
  \draw (0,0) rectangle (3,4);
  \node[fill=white,draw=black,circle] (0) at (1.5,0) {};
  \LB414 \LB324 \LB234 \RB215 \RB325 \RB335 \RB145
  \edge22 \edge33 \drop1 \throw1 \throw3 \throw4
  \RW{22}15 \RW{000}25 \RW{112}35 \RW{2}45
\end{scope}
\end{tikzpicture}
\begin{tikzpicture}
\begin{scope}[scale=.7]
  \draw (0,0) rectangle (3,4);
  \node[fill=white,draw=black,circle] (0) at (1.5,0) {};
  \LB215 \LB325 \LB335 \LB145 \RB215 \RB325 \RB335 \RB145
  \edge11 \edge22 \edge33 \edge44
  \RW{22}15 \RW{000}25 \RW{112}35 \RW{2}45
\end{scope}
\end{tikzpicture}
\begin{tikzpicture}
\begin{scope}[scale=.7]
  \draw (0,0) rectangle (3,4);
  \node[fill=white,draw=black,circle] (0) at (1.5,0) {};
  \LB215 \LB325 \LB335 \LB145 \RB415 \RB325 \RB135 \RB145
  \edge21 \edge32 \edge43 \drop1 \throw1 \throw4
  \RW{0004}15 \RW{112}25 \RW{2}35 \RW{4}45
\end{scope}
\end{tikzpicture}
\caption{A 4-card sequence \( (C_1,C_2,C_3,C_4) \) with the data
  \( (\alpha^{(0)}, \alpha^{(1)}, \alpha^{(2)}, \alpha^{(3)},
  \alpha^{(4)}) \) of when balls are thrown. Here \( \alpha^{(i-1)} \)
  and \( \alpha^{(i)} \) are shown
on the left and on the right, respectively, of each card \(C_i\).}
     \label{fig:embedding2}
\end{figure}

The \( \ell \)-card sequence \( (C_1,\dots,C_\ell) \) can be recovered
from \( (\alpha^{(0)},\dots,\alpha^{(\ell)}) \) because the
\( (b,k) \)-embedding of \( C_i \) is obtained from \( \alpha^{(i)} \)
by replacing every integer less than \( i \) by \( 0 \) and every
\( i \) by \( 1 \). Moreover, \( \alpha^{(i-1)} \) is also determined
by \( \alpha^{(i)} \) and the history of the balls thrown at beat
\( i \). More precisely, let \( \delta_i \) be the first word in
\( \alpha^{(i-1)} \) if there is at least one \( i \) in
\( \alpha^{(i)} \), and let \( \delta_i=\emptyset \) otherwise. Then
\( \alpha^{(i-1)} \) is determined by \( \alpha^{(i)} \) and
\( \delta_i \) as follows. If \( \delta_i=\emptyset \), then
\( \alpha^{(i-1)} = \alpha^{(i)} \). Otherwise, \( \alpha^{(i-1)} \)
is obtained from \( \alpha^{(i)} \) by deleting all \( i \)'s,
discarding all empty words if there are any, and adding \( \delta_i \)
at the beginning.

Applying this process iteratively, the whole sequence
\( (\alpha^{(0)},\dots,\alpha^{(\ell)}) \) is determined by the pair
\( (\gamma,\delta) \), where \( \gamma=\alpha^{(\ell)} \) and
\( \delta = (\delta_1,\dots,\delta_\ell) \). Observe that the length
of the word \( \delta_i \) is equal to the total number of \( i \)'s
in \( \alpha^{(i)} \), which is also equal to the total number of
\( i \)'s in \( \delta_{i+1},\dots,\delta_\ell \), and
\( \alpha^{(\ell)} \).

\begin{exam}
  The 4-card sequence in \Cref{fig:embedding2} corresponds to the pair
  \( (\gamma,\delta) \), where \( \gamma = (0004,112,2,4) \), and
  \( \delta=(0000,0011,\emptyset,22) \).
\end{exam}

The above observations imply that the \( \ell \)-card sequence
\( (C_1,\dots,C_\ell) \) can be identified with the pair
\( (\gamma,\delta) \). This leads us to the following definition.

\begin{defn}
  A \emph{\( (b,k,\ell) \)-embedding} is a pair \( (\gamma,\delta) \)
  of two sequences \( \gamma = (\gamma_1,\dots,\gamma_{r}) \) and
  \( \delta= (\delta_1,\dots,\delta_{\ell}) \) satisfying the
  following conditions:
  \begin{itemize}
  \item Each \( \gamma_i \) is a word of the form
    \( 0^{c_0} 1^{c_1} \cdots \ell^{c_\ell} \) for some integers
    \( c_0, c_1 ,\ldots, c_\ell \ge0 \) with
    \( 1\le c_0+c_1 + \cdots + c_\ell \le k \).
  \item The sum of the lengths of \( \gamma_i \) for all \( i\in [r] \)
    is equal to \( b \).
  \item Each \( \delta_i \) is a (possibly empty) word of the form
    \( 0^{d_0} 1^{d_1} \cdots (i-1)^{d_{i-1}} \) for some integers
    \( d_0,d_1 ,\ldots, d_{i-1} \ge0 \) such that
    \( 0\le d_0+d_1 + \cdots + d_{i-1}\le k \) and
    \( d_0+d_1 + \cdots + d_{i-1} \) is equal to the total number of
    \( i \)'s in \( \gamma_1 ,\dots,\gamma_r \) and
    \( \delta_{i+1},\delta_{i+2},\dots,\delta_{\ell} \).
  \end{itemize}
  Let \( \EE(b,k,\ell) \) be the set of \( (b,k,\ell) \)-embeddings.
\end{defn}

By the observations above, the map
\( (C_1,\dots,C_\ell)\mapsto (\gamma,\delta) \) is a bijection from
\( \JJ(b,k,\ell) \) to \( \EE(b,k,\ell) \). In order to enumerate
\( \EE(b,k,\ell) \) we need some definitions.

For a nonnegative integer \( n \), the \emph{complete homogeneous symmetric
  function} \(h_n(x_1,\dots,x_m)\) is defined by
\[
  h_n(x_1, \dots, x_m) = \sum_{i_1 +\cdots +i_m=n} x_1^{i_1} \cdots x_m^{i_m},
\]
where the sum is over all \( m \)-tuples \( (i_1,\dots,i_m) \) of
nonnegative integers summing to \( n \). Note that
we have \( h_0(x_1, \dots, x_m) = 1 \).

\begin{defn}\label{def:1}
  For indeterminates \( z_1,\dots,z_k \), we denote by
  \( D_{z_1,\dots,z_{k}} \) the linear operator on the space of
  formal power series in \( z_k \) defined by
\[
  D_{z_1,\dots,z_{k}} z_k^n
  = h_n(1,z_1,\dots,z_{k-1}) z_k^n.
\]
\end{defn}

Note that this can be seen as a generalization of the
\emph{\( q \)-derivative operator} \( (\frac{d}{dz})_q \), which is the
linear operator on the space of formal power series in \( z \) defined
by
\[
  \left( \frac{d}{dz} \right)_q z^n  = (1+q + \cdots + q^{n-1}) z^{n-1}.
\]
Hence \( D_{q,z} \) is equal to the operator \( (\frac{d}{dz})_q z  \),
which multiplies \( z \) and then takes the \( q \)-derivative.

Now we are ready to find an expression for the generating function for
\( J(b,k,\ell) \) when \( k \) and \( \ell \) are fixed.

\begin{thm}\label{thm:3}
  For fixed positive integers \( k \) and \( \ell \), we have
\begin{multline*}
  \sum_{b\ge0} J(b,k,\ell) x^b\\
  = [z_1^k \cdots z_\ell^k]
  \left(   \frac{1}{1-z_1} \cdots \frac{1}{1-z_\ell}
  D_{z_1,z_2}D_{z_1,z_2,z_3} \cdots D_{z_1,\dots,z_\ell} \frac{1}{2-h_k(1,x,xz_1,\dots,xz_\ell)}
 \right) .
\end{multline*}
\end{thm}

\begin{proof}
  We proceed similarly as in the proof of \Cref{prop:1}.
  Let \( W \) be the set of words of the form
  \( 0^{c_0} 1^{c_1} \cdots \ell^{c_\ell} \) for some integers
  \( c_0, c_1 ,\ldots, c_\ell \ge0 \) with
  \( 1\le c_0+c_1 + \cdots + c_\ell \le k \). For a word
  \( w= 0^{c_0} 1^{c_1} \cdots \ell^{c_\ell} \in W \), we define
  \( \len(w) = c_0 + \cdots + c_\ell \) and
  \[
    \wt(w) = z_1^{c_1} \cdots z_\ell^{c_\ell}.
  \]
  Then we have
  \[
    \sum_{w\in W} x^{\len(w)} \wt(w) = \sum_{1\le c_0 + \cdots + c_\ell\le k}
  x^{c_0 + \cdots + c_\ell} z_1^{c_1} \cdots z_\ell^{c_\ell}
  = h_k(1,x,xz_1,\dots,xz_\ell) -1.
  \]
  Therefore we obtain
  \begin{equation}\label{eq:A}
   A:= \frac{1}{2-h_k(1,x,xz_1,\dots,xz_\ell)} = 1+ \sum_{r\ge1} \sum_{(w_1,\dots,w_r)\in W^r}
    x^{\len(w_1) + \cdots + \len(w_r)} \wt(w_1) \cdots \wt(w_r).
  \end{equation}

  First, we consider how the operator \( D_{z_1,\dots,z_\ell} \) acts
  on a monomial in the right-hand side of \eqref{eq:A}. For
  \( (w_1,\dots,w_r)\in W^r \), we have
  \[
    x^{\len(w_1) + \cdots + \len(w_r)}  \wt(w_1) \cdots \wt(w_r)
    = x^b z_1^{n_1} \cdots z_\ell^{n_\ell}
  \]
  for some integers \( b \ge1 \) and \( n_1,\dots,n_\ell\ge0 \). Then
  \begin{align*}
    D_{z_1,\dots,z_\ell} x^b z_1^{n_1} \cdots z_\ell^{n_\ell}
    &=  x^b z_1^{n_1} \cdots z_\ell^{n_\ell} h_{n_\ell}(1,z_1,\dots,z_{\ell-1})\\
    &= x^b z_1^{n_1} \cdots z_\ell^{n_\ell} \sum_{c_0 + \cdots + c_{\ell-1}=n_\ell}z_1^{c_1} \cdots
    z_{\ell-1}^{c_{\ell-1}} .
  \end{align*}
  This implies that
  \begin{equation}\label{eq:A2}
    D_{z_1,\dots,z_\ell} A
    = 1 + \sum_{(w_1,\dots,w_r,u_\ell)\in X_\ell}  
    x^{\len(w_1) + \cdots + \len(w_r)}  \wt(w_1) \cdots \wt(w_r) \wt(u_\ell),
  \end{equation}
  where \( X_\ell \) is the set of tuples \( (w_1,\dots,w_r,u_\ell) \)
  such that \( w_i\in W \) for \( i\in [r] \) and \( u_\ell \) is a
  word of the form \( 0^{c_0} \cdots (\ell-1)^{c_{\ell-1}} \)
  for some nonnegative integers \( c_0,\dots,c_{\ell-1} \) with the
  condition that \( c_0 + \cdots + c_{\ell-1} \) is equal to the total
  number of \( \ell \)'s in \( w_1,\dots,w_r \).

  Next, we consider how the operator \( D_{z_1,\dots,z_{\ell-1}} \) acts
  on a monomial in the right-hand side of \eqref{eq:A2}. For
  \( (w_1,\dots,w_r,u_\ell)\in X_\ell \), we have
  \[
    x^{\len(w_1) + \cdots + \len(w_r)}  \wt(w_1) \cdots \wt(w_r) \wt(u_\ell)
    = x^b z_1^{n_1} \cdots z_\ell^{n_\ell}
  \]
  for some integers \( b \ge1 \) and \( n_1,\dots,n_\ell\ge0 \).
  Then we have
  \begin{align*}
    D_{z_1,\dots,z_{\ell-1}} x^b z_1^{n_1} \cdots z_\ell^{n_\ell}
    &=  x^b z_1^{n_1} \cdots z_\ell^{n_\ell} h_{n_{\ell-1}}(1,z_1,\dots,z_{\ell-2})\\
    &= x^b z_1^{n_1} \cdots z_\ell^{n_\ell} \sum_{c_0 + \cdots + c_{\ell-2}=n_{\ell-1}}z_1^{c_1} \cdots
    z_{\ell-1}^{c_{\ell-1}} .
  \end{align*}
  This implies that
  \begin{multline*}
    D_{z_1,\dots,z_{\ell-1}} D_{z_1,\dots,z_\ell} A\\
    = 1 + \sum_{(w_1,\dots,w_r,u_\ell,u_{\ell-1})\in X_{\ell-1}}  
    x^{\len(w_1) + \cdots + \len(w_r)}  \wt(w_1) \cdots \wt(w_r) \wt(u_\ell) \wt(u_{\ell-1}),
  \end{multline*}
  where \( X_{\ell-1} \) is the set of tuples
  \( (w_1,\dots,w_r,u_\ell,u_{\ell-1}) \) such that
  \( (w_1,\dots,w_r,u_\ell)\in X_{\ell} \) and \( u_{\ell-1} \) is a
  word of the form \( 0^{c_0} \cdots (\ell-2)^{c_{\ell-2}} \) for some
  nonnegative integers \( c_0,\dots,c_{\ell-2} \) with the condition
  that \( c_0 + \cdots + c_{\ell-2} \) is equal to the total number of
  \( (\ell-1) \)'s in \( w_1,\dots,w_r, u_\ell \).

  Applying the above argument iteratively, we obtain that
  \begin{multline}\label{eq:A3}
    D_{z_1,z_2} \cdots D_{z_1,\dots,z_\ell} A\\
    = 1 + \sum_{(w_1,\dots,w_r,u_\ell,\dots,u_1)\in X_1}
    x^{\len(w_1) + \cdots + \len(w_r)}  \wt(w_1) \cdots \wt(w_r) \wt(u_\ell) \cdots \wt(u_1),
  \end{multline}
  where \( X_1 \) is the set of tuples
  \( (w_1,\dots,w_r,u_\ell,\dots,u_1) \) such that for each
  \( i\in [r] \), \( w_i\in W \) and for each \( j\in[\ell] \),
  \( u_{j} \) is a word of the form
  \( 0^{c_0} \cdots (j-1)^{c_{j-1}} \) for some nonnegative integers
  \( c_0,\dots,c_{j-1} \) with the condition that
  \( c_0 + \cdots + c_{j-1} \) is equal to the total number of
  \( j \)'s in \( w_1,\dots,w_r, u_\ell,\dots,u_{j+1} \). By equation
  \eqref{eq:A3}, we have
  \begin{equation}\label{eq:8}
     [z_1^k \cdots z_\ell^k]
  \left(   \frac{1}{1-z_1} \cdots \frac{1}{1-z_\ell}
  D_{z_1,z_2}D_{z_1,z_2,z_3} \cdots D_{z_1,\dots,z_\ell} A
\right) = 1 + \sum_{b\ge1} |Y_b| x^b,
  \end{equation}
  where \( Y_b \) is the set of tuples
  \( (w_1,\dots,w_r,u_\ell,\dots,u_1)\in X_1 \) such that
  \[
    x^{\len(w_1) + \cdots + \len(w_r)}\wt(w_1) \cdots \wt(w_r) \wt(u_\ell) \cdots \wt(u_1) = 
   x^b z_1^{n_1} \cdots z_\ell^{n_\ell}
  \]
  for some integers \( 0\le n_1,\dots,n_\ell\le k \).

  It is immediate from the definitions of \( Y_b \) and
  \( \EE(b,k,\ell) \) that the map 
  \[
    (w_1,\dots,w_r,u_\ell,\dots,u_1)\mapsto (\gamma,\delta),
  \]
 where
  \( \gamma=(w_1,\dots,w_r) \) and \( \delta=(u_1,\dots,u_\ell) \), is
  a bijection from \( Y_b \) to \( \EE(b,k,\ell) \).
  Therefore we conclude
  \[
    1 + \sum_{b\ge1} |Y_b| x^b = 1 + \sum_{b\ge1} |\EE(b,k,\ell)| x^b
    = \sum_{b\ge0} J(b,k,\ell)x^b, 
  \]
 which completes the proof. 
\end{proof}

\section{Rationality of the generating function}
\label{sec:rationality}

In this section, as a consequence of \Cref{thm:3}, we show that the
generating function for \( J(b,k,\ell) \) with fixed \( k \) and
\( \ell \) is a rational function. This is equivalent to the statement
that the sequence \( \{J(b,k,\ell)\}_{b\ge0} \) satisfies a linear
recurrence relation; see \cite[Theorem~4.1.1]{EC1}.

We first review some basic properties of derivatives and
\( q \)-derivatives and extend these properties to the operator
\( D_{z_1,\dots,z_m} \). By the quotient rule in calculus, one can
easily deduce that the derivative of a rational function is also a
rational function. For the \( q \)-derivative, it is well known
\cite[Equation (11.4.1)]{Ismail} that
\[
 \left(\frac{d}{d x}\right)_q f(x)=\frac{f(x)-f(qx)}{x-qx}.
\]
This implies that the \( q \)-derivative of a rational function is
also a rational function. Note that since
\[
  D_{z_1,z_2} z_2^n = (1+z_1 + \cdots + z_1^{n}) z_2^n
  = \frac{(1-z_1^{n+1})z_2^n}{1-z_1},
\]
we have
\begin{equation}\label{eq:9}
  D_{z_1,z_2}f(z_2) = \frac{f(z_2)-z_1f(z_1z_2)}{1-z_1}.
\end{equation}
Hence, if \( f(z_2) \) is a rational function in \( z_2 \), then so is
\( D_{z_1,z_2}f(z_2) \).

Our strategy is to show that the operator \( D_{z_1,\dots,z_m} \), for
any \( m\ge2 \), also preserves the rationality of a formal power
series. To do this, we need the following two lemmas.

\begin{lem}\label{lem:1}
  For integers \( n\ge0 \) and \( m\ge2 \), we have
\[
  h_n(1,z_1,\dots,z_{m}) =
  \frac{z_{m-1}h_n(1,z_1,\dots,z_{m-1}) - z_m h_n(1,z_1,\dots,z_{m-2},z_{m})}{z_{m-1}-z_{m}}.
\]
\end{lem}

\begin{proof}
  Let \( \mathbb{R}[z_1,\dots,z_k] \) denote the space of polynomials in the variables \( z_1,\dots,z_k \).
 Consider the linear operator \(\Delta: \mathbb{R}[x] \longrightarrow \mathbb{R}[x, y]\) defined by
\[
\Delta(p(x))=\frac{x \cdot p(x)-y \cdot p(y)}{x-y} .
\]
Note that \(\Delta\left(x^n\right)=\sum_{i+j=n} x^i y^j\). Hence, by
letting \( x=z_{m-1} \) and \( y=z_m \) and applying this operator to
\(h\left(z_1, \ldots, z_{m-1}\right)\), we obtain the desired
identity.
\end{proof}

\begin{lem}\label{lem:2}
  For an integer \( m\ge2 \), we have
  \[
  D_{z_1,\dots,z_{m+1}}
  = \frac{z_{m-1}D_{z_1,\dots,z_{m-1},z_{m+1}} - z_m D_{z_1,\dots,z_{m-2},z_m,z_{m+1}}}{z_{m-1}-z_{m}} .
\]
\end{lem}

\begin{proof}
  Both sides are linear operators in the space of formal power series
  in \( z_{m+1} \). Therefore it suffices to show that they act on
  \( z_{m+1}^n \) in the same way, that is,
  \[
    D_{z_1,\dots,z_{m+1}} z_{m+1}^n
  = \frac{z_{m-1}D_{z_1,\dots,z_{m-1},z_{m+1}} - z_m D_{z_1,\dots,z_{m-2},z_m,z_{m+1}}}{z_{m-1}-z_{m}} z_{m+1}^n.  
  \]
  But this is equivalent to
  \[
  h_n(1,z_1,\dots,z_{m})z_{m+1}^n =
  \frac{z_{m-1}h_n(1,z_1,\dots,z_{m-1}) - z_m h_n(1,z_1,\dots,z_{m-2},z_{m})}{z_{m-1}-z_{m}} z_{m+1}^n,
  \]
  which follows from \Cref{lem:1}.
\end{proof}

Now we are ready to show that the operator \( D_{z_1,\dots,z_{m}} \)
preserves the rationality of a formal power series.

\begin{prop}\label{prop:d_rational}
  Suppose that \( \ell \) and \( m \) are integers with
  \( 2\le m\le \ell \) and let \( z_1,\dots,z_\ell \) be
  indeterminates. If \( f(z_1,\dots,z_\ell) \) is a formal power
  series in \( z_1,\dots,z_\ell \) that is a rational function in
  \( z_1,\dots,z_\ell \), then
  \( D_{z_1,\dots,z_{m}} f(z_1,\dots,z_\ell) \) is a rational function
  in \( z_1,\dots,z_\ell \).
\end{prop}

\begin{proof}
  If \( m=2 \), the statement follows from \eqref{eq:9}. Suppose that
  the statement holds for \( m\ge2 \). Then by \Cref{lem:2} the case
  of \( m+1 \) also holds. The proof then follows by induction.
\end{proof}

Finally, we can prove the rationality of the generating function
studied in the previous section.

\begin{cor}\label{cor:1}
  For fixed positive integers \( k \) and \( \ell \), the generating function
\[
  \sum_{b\ge0} J(b,k,\ell) x^b
\]
is a rational function in the variable \( x \).
\end{cor}

\begin{proof}
  By \Cref{thm:3} and \Cref{prop:d_rational}, we have
  \begin{equation}\label{eq:10}
    \sum_{b\ge0} J(b,k,\ell) x^b = [z_1^k \cdots z_\ell^k]
    f(z_1,\dots,z_\ell,x)
\end{equation}
for a formal power series \( f(z_1,\dots,z_\ell,x) \) in
\( z_1,\dots,z_\ell,x \) that is a rational function in these
indeterminates. If \( g(z) \) is a formal power series in \( z \)
that is a rational function in \( z \) and some other indeterminates,
say \( u_1,\dots,u_r \), then by the quotient rule,
\( [z^k]g(z)=k! g^{(k)}(0) \) is a rational function in
\( u_1,\dots,u_r \). Therefore the right-hand side of \eqref{eq:10} is
a rational function in \( x \) as desired.
\end{proof}

By \Cref{cor:1}, for fixed positive integers \( k \) and \( \ell \),
the sequence \( \{ J(b,k,\ell) \}_{b\ge0} \) satisfies a linear
recurrence relation. However, due to the complexity of its generating
function formula in \Cref{thm:3}, finding an explicit recurrence
relation appears to be challenging. In \cite{Ekhad2021}, the authors
used holonomic methods to find recurrence relations for the number of
multiset derangements. It would be interesting to see if their method
can be applied to the above sequence.

\section{Conclusion}
\label{sec:conclusion}

In this paper we found an expression for the generating function
\begin{equation}\label{eq:12}
  \sum_{b\ge0} J(b,k,\ell) x^b
\end{equation}
for the number of multiplex juggling card sequences when the capacity
\( k \) and the length \( \ell \) are fixed. As a consequence, we
showed that this generating function is a rational function in
\( x \). Equivalently, the sequence \( \{J(b,k,\ell)\}_{b\ge0} \)
satisfies a linear recurrence relation.

Note that there are three parameters in the number \( J(b,k,\ell) \)
and the generating function in \eqref{eq:12} keeps track of \( b \).
Therefore it is natural to consider the following two generating
functions:
\begin{align}
\label{eq:13}  \sum_{k\ge0} J(b,k,\ell) y^k,\\
  \label{eq:14} \sum_{\ell\ge0} J(b,k,\ell) z^\ell.
\end{align}
Since \( J(b,k,\ell) = J(b,\infty,\ell) \) for \( k\ge b \), it is
immediate that the generating function \eqref{eq:13} is a rational
function in \( y \). Using the transfer matrix method
\cite[Section~4.7]{EC1}, one can show that the generating function in
\eqref{eq:14} is also a rational function in \( z \). It would be very
interesting to see if this rationality continues to hold for the
multivariate generating function keeping track of all three parameters
\( b,k \), and \( \ell \).

\begin{problem}
  Determine whether the following is a rational function in the three variables \( x,y,z \):
  \[
  \sum_{b\ge0} \sum_{k\ge0}\sum_{\ell\ge0} J(b,k,\ell) x^by^k z^\ell.
\]
\end{problem}

Note that \( J(b,k,\ell) \) can be used to compute the number of ways
to juggle \( b \) balls with capacity \( k \) for beats
\( 1,\dots,\ell \) without any restrictions on the initial and final
states of the balls. In order to enumerate periodic multiplex juggling
patterns we need to consider the number \( J_0(b,k,\ell) \) of
\( \ell \)-card sequences \( (C_1,\dots,C_\ell)\in \JJ(b,k,\ell) \) such that
the departure composition of \( C_1 \) is equal to the arrival
composition of \( C_\ell \). It would be interesting to extend our
results to \( J_0(b,k,\ell) \). We end this paper with the following
problems.

\begin{problem}
  For fixed \( k \) and \( \ell \), find a formula for the generating function
  \[
  \sum_{b\ge0} J_0(b,k,\ell) x^b.
\]
\end{problem}

\begin{problem}
  Determine whether the following is a rational function in the three variables \( x,y,z \):
  \[
  \sum_{b\ge0} \sum_{k\ge0}\sum_{\ell\ge0} J_0(b,k,\ell) x^by^k z^\ell.
\]
\end{problem}

\section*{Acknowledgments}

We are grateful to the anonymous referee for the valuable comments,
especially for the short proof of \Cref{lem:1}. This project was
initiated as part of the research and education (R\&E) program at
Gyeonggi Science High School for the Gifted (GSHS) in 2023.

\bibliographystyle{abbrv}

\end{document}